\documentclass[a4paper]{amsart}
\usepackage{amsthm,amssymb,amsmath,graphicx,mathtools,enumerate}
\usepackage{caption}
\usepackage{subcaption}
\usepackage[shortlabels]{enumitem} 
\usepackage[british]{babel}
\usepackage[utf8x]{inputenc}
\usepackage{hyperref}
\usepackage{soul}

\usepackage[usenames,dvipsnames]{color}

\theoremstyle{plain}
\newtheorem{theorem}{Theorem}[section]

\newtheorem{lemma}[theorem]{Lemma}

\newtheorem{claim}[theorem]{Claim}

\newtheorem{fact}[theorem]{Fact}

\theoremstyle{definition}
\newtheorem{definition}[theorem]{Definition}

\numberwithin{equation}{section}

\newenvironment{proofclaim}[1][Proof of the claim]{\begin{proof}[#1]}{\end{proof}}

\newcommand{\mylabel}[2]{#2\def\@currentlabel{#2}\label{#1}}
\newcommand{\nc}[2]{{#2}}

\newcommand{\eps}{\varepsilon}
\newcommand{\sm}{\setminus}
\renewcommand{\subset}{\subseteq}
\newcommand{\Bin}{\text{Bin}}
\newcommand{\Gnp}{G({n,p})}

\newcommand{\cC}{\mathcal{C}}
\newcommand{\cL}{\mathcal{L}}
\newcommand{\cX}{\mathcal{X}}
\newcommand{\cV}{\mathcal{V}}

\newcommand{\tref}[1]{Theorem~\ref{#1}}

\newcommand{\G}{\Gamma}
\newcommand{\Vexc}{V^{\text{exc}}}

\newcommand{\Vdeg}{V^{\deg}}
\newcommand{\Vinh}{V^{\text{inh}}}
\newcommand{\tcX}{\tilde{\mathcal{X}}}
\newcommand{\tX}{\tilde X}

\renewcommand{\c}[1]{\overline{#1}}

\def\COMMENT#1{}
	\let\COMMENT=\footnote          

\begin{document}
\date{\today}
\title{Monochromatic cycle partitions in random graphs}
\author{Richard Lang}
\address[R.~Lang]{Institute for Computer Science \\ Heidelberg University \\   69120 Heidelberg, Germany}
\author{Allan Lo}
\address[A.~Lo]{School of Mathematics\\ University of Birmingham\\ Birmingham B15 2TT\\ United Kingdom}
\email{lang@informatik.uni-heidelberg.de, s.a.lo@bham.ac.uk}
\thanks{The research leading to these results was supported by EPSRC, grant no.  EP/P002420/1 and partially by the Deutsche Forschungsgemeinschaft (DFG, German Research Foundation) -- 428212407 (R. Lang)}

\begin{abstract}
Erd\H{o}s, Gy\'arf\'as and Pyber showed that every $r$-edge-coloured complete graph $K_n$ can be covered by $25 r^2 \log r$ vertex-disjoint monochromatic cycles (independent of $n$).
Here, we extend their result to the setting of binomial random graphs. 
That is, we show that if $p  = p(n) \nc{\geq}{=} \Omega(n^{-1/(2r)})$, then with high probability any $r$-edge-coloured $G(n,p)$ can be covered by at most $1000 r^4 \log r $ vertex-disjoint monochromatic cycles. 
This answers a question of Kor\'andi, Mousset, Nenadov, \v{S}kori\'{c} and Sudakov. 
\end{abstract}

\maketitle

\section{Introduction}
An active line of current research concerns sparse random analogues of combinatorial
theorems.
An early example of this type of result was given by Rödl and Ruciński~\cite{RR95}, who proved a random analogue of Ramsey's theorem.
Similar results have been obtained for asymmetric and hypergraph Ramsey problems (see the survey of Conlon~\cite{Con14}).
In this paper we investigate the random analogue for monochromatic cycle partitioning.

Given an edge-coloured graph, how many vertex-disjoint monochromatic cycles are necessary to cover its vertices?
Such a collection of cycles is called a \emph{monochromatic cycle partition}.\footnote{For this paper, we allow a single vertex and an edge to be `degenerate' cycles.} The problem of finding a partition with few cycles was first investigated for edge-coloured complete graphs.
Erd\H{o}s, Gy\'arf\'as and Pyber~\cite{EGP91} proved that there is a function $f(r)$ such that every $r$-edge-coloured complete graph~$K_n$ admits a partition into at most $f(r)$ monochromatic cycles (independent of $n$). 
In particular, they showed that $f(r) \le 25 r^2 \log r $ and further conjectured that $f(r) = r$. 
The case $r=2$ \nc{was also}{had already been} conjectured by Lehel in~1979.
\L{}uczak, R\"odl and Szemer\'{e}di~\cite{LRS98} showed that $f(2) =2$ for large~$n$.
The bound on~$n$ was then reduced by Allen~\cite{ABH+16} (though still large).
Finally, Bessy and Thomass\'e~\cite{BT10} showed that $f(2) = 2$ for all~$n$ by an elegant argument. 
For $r = 3$, Gy\'arf\'as, Ruszink\'o, S\'ark\"ozy and Szemer\'edi~\cite{GRSS06} showed that $f(3) \le 17$.
For general~$r$, the best known upper bound for $f(r)$ is due to Gy\'arf\'as, Ruszink\'o, S\'ark\"ozy and Szemer\'edi~\cite{GRSS11}, who showed that $f(r) \le 100 r \log r$ (for large~$n$). 
On the other hand, Pokrovskiy~\cite{Pok14} disproved Erd\H{o}s, Gy\'arf\'as and Pyber's conjecture by showing that $f(r) > r$ for all $r \ge 3$.
The question of whether $f(r)$ is linear in~$r$ is still open.
There has also been a considerable interest in monochromatic cycle partitions of host graphs that are not complete.
For instance, graphs with few missing edges~\cite{GJS97}; small independence number~\cite{Sar11}; large minimum degree~\cite{BBG+14,DN17,Let15} and bipartite graphs~\cite{Hax97,LS16}. 
See a survey of Gy\'arf\'as~\cite{Gya16} for further information.

In this paper, we consider monochromatic cycle partitions in binomial random graphs~$\G \sim \Gnp$.
The study of partitioning $\G \sim \Gnp$ into monochromatic subgraphs was \nc{first}{} initiated  by Bal and DeBiasio~\cite{BD17}, who showed that if $p  \geq  C (\log n / n)^{1/3}$, then with high probability (w.h.p.), any $2$-edge-coloured $\G \sim \Gnp$ admits a partition into two vertex-disjoint monochromatic trees. 
Recently, Kohayakawa, Mota and Schacht~\cite{KMS16} showed that the same holds for $p = \omega ( (\log n / n)^{1/2} )$. 
For more colours, Ebsen, Mota and Schnitzer (see~\cite[Proposition~4.1]{KMS16}) showed that w.h.p. there exists an $r$-edge-colouring of $\G \sim \Gnp$ with $p \ge (\log n / n )^{1/(r+1)}$\nc{,}{} which cannot be partitioned by $r$ vertex-disjoint monochromatic trees.
Covering $\G \sim \Gnp$ by (not necessarily vertex-disjoint) monochromatic cycles was studied by Kor\'andi, Mousset, Nenadov, \v{S}kori\'{c} and Sudakov~\cite{KMN+17}, who showed that if $p \ge  n^{-1/r + \eps}$, then w.h.p. any $r$-edge-coloured $\G \sim \Gnp$ can be covered by $O(r^8 \log r)$ monochromatic cycles. 
The same authors asked whether one can prove a random analogue of Erd\H{o}s, Gy\'arf\'as and Pyber's theorem, that is, any $r$-edge-colouring of $\G \sim \Gnp$ admits a partition into constantly many monochromatic cycles.
In this paper, we give an affirmative answer \nc{with}{for} $p = \Omega( n^{-1/(2r)})$.

\begin{theorem} \label{thm:main}
	Let $r \geq 2$ and $p  = p(n) \geq 2^9  r^{5} n^{-1/(2r)}$.
	Then w.h.p. the random graph $\G \sim \Gnp$ satisfies the following property.
	Any $r$-edge-colouring of $\G$ admits a partition into at most $1000 r^4 \log r$ monochromatic cycles.
\end{theorem}

It would be interesting to improve our bound on $p$.
A construction of Bal and DeBiasio~\cite{BD17} shows that for $p = o((r \log n /n)^{1/r})$ w.h.p. there exists an $r$-edge-colouring of $\G \sim G(n, p)$, which
requires an unbounded number of monochromatic components (and in particular, cycles)
to cover all vertices.
In light of this, it seems natural to conjecture the threshold to be of order $(\log n/n)^{1/r}$.

We did not attempt to optimize the number of cycles needed in Theorem~\ref{thm:main}.
Thus it is likely that our bound offers some room for improvement.
Let $f_{n,p}(r)$ be the minimum number of cycles needed such that w.h.p. every $r$-edge-colouring of $\G \sim G(n,p)$ admits a partition into at most $f_{n,p}(r)$ monochromatic cycles.
Bal and DeBiasio~\cite{BDPC} showed that $f_{n,p}(2) > 2$ if $p\leq 1/2$.
Moreover, Kor{\'a}ndi, Lang, Letzter and Pokrovskiy~\cite{KLLP19} recently constructed $r$-edge-coloured graphs on~$n$ vertices with minimum degree $(1-\eps)n$ which cannot be partitioned into fewer than $\Omega(\eps^2 r^2)$ monochromatic cycles.
Thus together with our conjectured threshold, an immediate question would be if it is true that $f_{n,p}(r)= o(r^2)$ for $p=O((\log n /n)^{1/r})$?

\nc{}{We note that since the appearance of this article, the last question has been answered negatively by Buci\'c, Kor{\'a}ndi and Sudakov~\cite{BKS19}.}

\section{Notation}
Let $G=(V,E)$ be a graph and $U,W \subseteq V$ be disjoint subsets of vertices. 
We denote the complement of $U$ in $G$ by $\c{U} = V \sm U$.
We write $G[U]$ for the subgraph of $G$ induced by $U$ and $G -U$ for $G [ \c{U} ]$.
We write $e_G(U,W)$ for the number of edges in~$G$ with one vertex in~$U$ and one in~$W$.
For a vertex $v \in V$, we write $N_G(v)$ for its neighbourhood in $G$ and $\deg_G(v)= |N_G(v)|$ for its degree.
We let $N_G(v,W)$ be the set of all neighbours of $v$ in $W$ in $G$ and $\deg_G(v,W) = |N_G(v, W)|$.
We denote $N_G^*(U,W) = \bigcap_{u \in U}N_G(u,W)$ and $\deg^*_G(U,W) = |N_G^*(U,W)|$.
We omit the subscript $G$ when it is clear from the context.
For a collection $\mathcal{C}$ of graphs, we write $V(\cC) = \bigcup_{C \in \cC} V(C)$.

For $a,b,c >0$, we write $a = b \pm c$ if $b - c \leq a \leq b+c$ and $a \neq b \pm c$ otherwise.
We write $\log a$ for the natural logarithm of $a$.
For sakes of exposition, we will omit ceiling and floor signs, whenever it is not important for the argument.

\section{Proof outline}

We sketch the proof of Theorem~\ref{thm:main}. 
Consider an $r$-edge-coloured $\G \sim \Gnp$. 
It is not difficult to find a small family $\mathcal{C}_{\text{almost}}$ of vertex-disjoint monochromatic cycles covering all but at most $\eps n$ vertices.
Indeed a standard approach, which was introduced by \L uczak~\cite{Luc99} and uses the (sparse) regularity lemma, allows us to reduce the problem of finding a large cycle in $\G$ to finding a large matching in the reduced graph $R$ of a regular partition of $\G$.
Thus we can obtain $\mathcal{C}_{\text{almost}}$ by finding a large matching in the union of $4r^2$ monochromatic components of $R$.

The main difficulty in proving Theorem~\ref{thm:main} is to cover these leftover vertices with vertex-disjoint cycles that are also disjoint from~$\mathcal{C}_{\text{almost}}$. 
Since these leftover vertices can only be determined after edge-colouring~$\G$, we cannot fix their location prior to exposing the edges.
\nc{}{We note that this is one {of the main differences} between covering and partitioning with monochromatic cycles.}
In the following lemma, we show that every small vertex set can be covered by a small number \nc{}{of} vertex-disjoint monochromatic cycles.
Moreover, almost all\nc{ of its}{} vertices \nc{}{used in this process} will be \nc{in}{taken from} a predetermined vertex set~$U$. 

\begin{lemma}\label{lem:absorption}
	Let $r \geq 2$, $0 < \beta < 1$ and $p = p(n) \geq 2^3  r^{5}(\beta ^2 n)^{-1/(2r)}$.
	Then w.h.p. the random graph $\G \sim \Gnp$ satisfies the following property.
		For any $r$-edge-colouring of $\G$ and any disjoint subsets of vertices $U$ and $W$ with $|U| \geq \beta n$ and $|W| \leq (\beta/(400r))^{4} n$, there exists a collection $\cC$ of at most $900 r^4 \log r$ disjoint monochromatic cycles such that $W \subseteq V(\cC)$ and $|V (\cC) \setminus (U \cup W) | \leq 48r^9/ (\beta p^r)$.
\end{lemma}

Note that, under the assumption of Lemma~\ref{lem:absorption}, there might be a small set $W' \subseteq	W$, whose neighbourhood does not behave in the expected way (e.g. there are too few edges between $W'$ and~$U \cup W$). 
Hence, in order to cover~$W'$, we will need to use some vertices outside $U$. 
We defer the proof of Lemma\nc{s}{}~\ref{lem:absorption} to Section~\ref{sec:5}.
Our proof of Lemma~\ref{lem:absorption} is based on arguments of Kor\'andi, Mousset, Nenadov, \v{S}kori\'{c} and Sudakov~\cite{KMN+17}.
However,\nc{ some}{} new ideas are needed to ensure vertex-disjoint cycles. 
Our proof of Lemma~\ref{lem:absorption} requires that $p = p(n) = \Omega(n^{-1/(2r)})$, which is the main reason for our bound in Theorem~\ref{thm:main}.

Recall that we require the cycles covering $W$ to be disjoint from those 
covering the rest of the vertices, i.e. $\mathcal{C}_{\text{almost}}$.
To deal with this, we ensure that $\mathcal{C}_{\text{almost}}$ is `robust'.
Roughly speaking, even after deleting a few vertices of $\mathcal{C}_{\text{almost}}$, there is a \nc{}{small} monochromatic cycle partition $\mathcal{C}'_{\text{almost}}$ on the remaining vertices.
This strategy was introduced by Erd\H{o}s, Gy\'arf\'as and Pyber~\cite{EGP91} and has become fairly standard in the area.
However, our ``robustness'' property is more general in the sense that we further allow the deletion of a small but arbitrary set outside of~$U$.
This is crucial for our approach as we use such a small but arbitrary set to cover~$W$ as stated Lemma~\ref{lem:absorption}.

\nc{}{The next lemma is used to find the family $\mathcal{C}_{\text{almost}}$.}

\begin{lemma}\label{lem:approximate-partition}
	Let $r \geq 2$ and $\eps_1 > 0$.
	Then there exists $C = C(r,\eps_1) >0$ such that, for $p = p(n) \geq C(\log n /n)^{1/(r+1)}$, w.h.p.  the random graph $\G \sim \Gnp$ satisfies the following property.
	For any $r$-edge-colouring of~$\G$, there are disjoint vertex sets $U$ and $W$ with $|U| \geq {2^{-12}}n$ and  $|W| \leq {\eps_1} n$ such that the following holds.
	For any sets $U' \subseteq U$ and $U^+ \subseteq V(\G)$ with $|U^+| \leq 2^{18}r^9/p^r$, the graph $\G - (W \cup U^+ \cup U')$ admits a partition into at most $4r^2+1$ monochromatic cycles.
\end{lemma}

We will prove Lemma~\ref{lem:approximate-partition} in Section~\ref{sec:6}. (See the beginning of Section~\ref{sec:6} for a sketch proof of Lemma~\ref{lem:approximate-partition}.)
We now prove Theorem~\ref{thm:main} using these two lemmas. 
\begin{proof}[Proof of Theorem~\ref{thm:main}]
	Let $\beta=2^{-12}$ and $\eps_1 = (\beta/(400r))^4$.
	Note that since $p \ge 2^9 r^{5} n^{-1/(2r)}$, w.h.p. $\G \sim \Gnp$ satisfies the conclusions of Lemmas~\ref{lem:absorption} and~\ref{lem:approximate-partition}.
	We will deduce the theorem from these properties.
	
	Consider any $r$-edge-colouring of $\G$.
	By Lemma~\ref{lem:approximate-partition}, there exist disjoint vertex sets $U$ and~$W$ with $|U| \geq \beta n$ and  $|W| \leq \eps_1 n$.
	Moreover, for any sets $U' \subseteq U$ and $U^+ \subseteq V(\G)$ with $|U^+| \leq  2^{18}r^9 /p^r$, the graph $\G - (W \cup U^+ \cup U')$ admits a partition into at most $4r^2+1$ monochromatic cycles.
	By Lemma~\ref{lem:absorption}, there exists a collection $\cC_1$ of at most $900 r^4 \log r$ monochromatic disjoint cycles such that $W \subseteq V(\cC_1)$ and $|V (\cC_1) \setminus (U \cup W) | \leq 48r^9/ (\beta p^r) \le 2^{18}r^9 /p^r$. 
	Let $U' = V(\cC_1) \cap U$ and $U^+ = V (\cC_1) \setminus (U \cup W)$, so $|U^+|\leq  2^{18}r^9 /p^r$.
	By the choice of $W$ and $U$, the graph $\G - (W \cup U^+ \cup U')$ admits a partition~$\cC_2$ into at most $4r^2+1$ monochromatic cycles.
	Thus $\cC_1 \cup \cC_2$ is a partition of $\G$ into at most $900r^4 \log r + 4r^2 +1 \leq 1000r^4 \log r$ monochromatic cycles.
\end{proof}

\section{Probabilistic tools}
The exposition of our probabilistic tools follows the one of \nc{}{Kor{\'a}ndi,  Mousset,  Nenadov,  {\v S}kori{\' c}  and Sudakov}~\cite{KMN+17}.
We will use the following Chernoff-type bounds on the tails of the
binomial distribution.
\begin{lemma}[{\cite[Theorem 2.1]{JLR01}}]\label{lem:che}
	Let  $0<\alpha<3/2$ and $X \sim \emph{\text{Bin}}(n,p)$ be a binomial random variable. 
	Then $\mathbb{P}\left(|X - np | > \alpha n p \right) < 2e^{-\alpha^2np/3}$.
\end{lemma}

\begin{lemma}[{\cite[Lemma 3.8]{KMN+17}}] \label{lem:densityXY}
	Fix $0< \alpha, \beta < 1$
	and let $C=6/(\alpha^2\beta)$ and
	$D=9/\alpha^2$. Then, for every $p = p(n) \in (0,1)$, w.h.p. the random
	graph $\G \sim \Gnp$ satisfies the following property. For any two disjoint
	subsets $X,Y\subseteq V(\G)$, satisfying either of
	\begin{enumerate}[\upshape(1)]
		\item $|X|, |Y| \ge D\log n/ p$, or
		\item $|X| \ge C / p$ and $|Y| \ge \beta n$,
	\end{enumerate}
	we have $e(X,Y) = (1 \pm \alpha) |X| |Y| p$.
\end{lemma}

\begin{lemma}[{\cite[Lemma 3.9]{KMN+17}}] \label{lem:density_triples}
	For every $p = p(n) \in (0,1)$, w.h.p. the random graph $\G \sim \Gnp$ satisfies
	the following property.
	For every family~$\cL$ of $\ell$ disjoint pairs of vertices and every set $Y$ of $3\ell$ vertices that is disjoint from each pair in~$\cL$, we have
	\begin{align*}
	\sum_{\{v,w\} \in \cL} \deg^*_G(\{v,w\}, Y) \le \begin{cases}
	72 \ell \log n &\text{if } \ell \le 6 \log n / p^2, \\
	6\ell^2 p^2 &\text{otherwise.}
	\end{cases}
	\end{align*}
\end{lemma}

The following lemma plays a key role in our proof. 
It says that given any \nc{}{sufficiently large} vertex set~$X$, there exists a \nc{vertex small}{small vertex} set~$Y$ such that every $k$-set\nc{s}{} $S \subseteq \overline{ X \cup Y}$ has the expected number of common \nc{numbers}{neighbours} in~$X$. 

\begin{lemma}\label{lem:dense-neighbours}
	Let $k \geq 1$, $\alpha, \beta  \in (0,1)$.
	Let $K = 12k/(\alpha^2\beta)$ and $p = p(n) \geq (K\log n/n)^{1/k}$.
	Then w.h.p. the random graph $\G \sim \Gnp$ satisfies the following property.
	For any vertex set~$X$ with $|X| \geq \beta n$, there exists a set $Y \subseteq \overline{X}$ of size at most $K/p^{k}$ such that all $k$-sets $S \subseteq \overline{ X \cup Y}$ have $\deg^*(S,X) =(1 \pm \alpha) p^k|X|$.
\end{lemma}

\begin{proof}
	We first show that w.h.p. $\G \sim \Gnp$ satisfies
	\begin{enumerate}[label={($\ast$)}]
		\item \label{itm:dense-neighours-1}
			for any set~$X$ of at least $\beta n$ vertices and any family~$M$ of at least $K/(2kp^k)$ disjoint $k$-sets in~$\c{X}$, \nc{}{it holds that}
	$	\sum_{S \in M} \deg^*(S,X)  =  (1 \pm \alpha) p^k |M||X|$.
	\end{enumerate}

	To see this consider a set~$X$ with $|X| \geq \beta n$ and a family~$M$ of at least $K/(2kp^k)$ disjoint $k$-sets in~$\c{X}$.
	Let $Z = \sum_{S \in M} \deg^*(S,X)$, that is, $Z$ counts the number of pairs~$(S,x) \in M \times X$ for which $x \in N^*(S)$.
	For each $(S,x) \in M\times X$, the probability that $x \in N^*(S)$ is $p^k$.
	Since the events $x \in N^*(S,X)$ and $x' \in N^*(S',X)$ are independent for $(S,x) \neq (S',x')$, it follows that $Z \sim \Bin(|M||X|,p^k)$.
	Thus Lemma~\ref{lem:che} implies that
	\begin{align*}
	\mathbb{P}\nc{[}{(}Z \neq (1 \pm \alpha) p^k |M| |X| \nc{]}{)}
	&\leq 2\exp\left({-\frac{\alpha^2}{3}p^k|M||X|}\right) 
	\leq 2\exp\left({-\frac{\alpha^2\beta}{3}p^kn|M|}\right).
	\end{align*}
	By taking the union bound of these events over all such sets $X$ and families $M$, we see that the probability that \ref{itm:dense-neighours-1} fails is at most 
	\begin{align*}
	&\quad \sum_{m = K/(2kp^k)}^{n/k} n^{km}\cdot 2^{n}\cdot 2 \exp\left({-\frac{\alpha^2\beta}{3}p^knm}\right)
	\\&\leq \sum_{m = K/(2kp^k)}^{n/k} \exp\left(km\log n + n{-\frac{\alpha^2\beta}{3}p^knm}\right)
	\\&= \sum_{m = K/(2kp^k)}^{n/k} \exp \left( \left(\log n + \frac{n}{km} {-\frac{\alpha^2\beta}{3k}p^kn}\right) {km} \right)
	\\&\leq \sum_{m = K/(2kp^k)}^{n/k} \exp\left(\left(\log n + \left(\frac{2}{K} {-\frac{\alpha^2\beta}{3k}}\right)p^kn\right) {km}\right)
	\\&\leq \sum_{m = K/(kp^k)}^{n/k} n^{-km} \le n^{-1},
	\end{align*}
	where the penultimate inequality uses the fact that ${\alpha^2\beta}/{3k} = 4/K $ and $p^k \geq K\log n/n$.
	Thus we can assume that $\G$ satisfies~\ref{itm:dense-neighours-1}.
	We will deduce the lemma from this property.
	
	Suppose that we are given a vertex set~$X$ with $|X| \geq \beta n$.
	Let $H^-$ be an auxiliary $k$-uniform hypergraph on vertex set~$\c{X}$ such that a $k$-set~$S$ is an edge in~$H^-$ if $\deg^*(S,X) < (1-\alpha) p^k |X|$.
	Similarly, let $H^+$ be an auxiliary $k$-uniform hypergraph on vertex set~$\c{X}$ such that a $k$-set~$S$ is an edge in~$H^+$ if $\deg^*(S,X) > (1+\alpha) p^k |X|$.
	Let $M^-$ and $M^+$ be matchings of maximum size in $H^-$ and~$H^+$, respectively.
	Then \ref{itm:dense-neighours-1} implies that $|M^-|+|M^+| \leq K /(kp^{k})$.
	Let $Y = \bigcup_{S \in M^- \cup M^+}S$ and note that $|Y| \leq K/p^k$.
	By maximality of $M^-$ and $M^+$, every $k$-set $S$ in $\c{X\cup Y}$ satisfies $\deg^*(S,X) =(1 \pm \alpha) p^k|X|$, as desired.
\end{proof}

\section{Proof of Lemma~\ref{lem:absorption}} \label{sec:5}

In this section we prove Lemma~\ref{lem:absorption}, that is, every small vertex set can be covered by few monochromatic cycles.
We start by setting up a few auxiliary lemmas.

\begin{theorem}[Sárközy~\cite{Sar11}]\label{thm:sarkozy}
	Let $G$ be a graph of independence number $\alpha$.
	Then any $r$-edge-colouring of $G$ admits a partition into at most $25(\alpha r)^2 \log (\alpha r)$ monochromatic cycles.
\end{theorem}

\begin{lemma}[{\cite[Lemma~2.1]{KMN+17} }]\label{lem:KMN+17}
	For every $r \geq 2$, $\beta >0$, $K = 4000r^4/\beta$, $C = (100r/\beta)^8$ and $n_0$ sufficiently large, the following holds.
	Let $G$ be a graph on $n \geq n_0$ vertices satisfying the following properties for $p   \geq C(\log n /n)^{1/2}$.
	\begin{enumerate}[\upshape(i)]
		\item \label{itm:approx-abs-edges} For any two disjoint $X,Y \subseteq V(G)$, such that $|X| \geq 192r/(\beta p)$ and $|Y|\geq  \beta n/(2r)$, it holds that $e(X,Y) \leq 5p|X||Y|/4$.
		\item \label{itm:approx-abs-triples} For every family~$\cL$ of $\ell$ disjoint pairs of vertices and every set $Y$ of~$3\ell$ vertices that is disjoint from each pair in~$\cL$, we have
		\begin{align*} 
		\sum_{\{v,w\} \in \cL} \deg^*(\{v,w\}, Y) \le \begin{cases}
		72 \ell \log n &\text{if } \ell \le 6 \log n / p^2,\\
		6\ell^2 p^2 &\text{otherwise.}
		\end{cases}
		\end{align*}
	\end{enumerate}
	Let $U,W \subseteq V(G)$ be disjoint with $|U| \geq \beta n$ and $|W| \leq \beta^4n/(100r)^4$ and
	\begin{enumerate}[\upshape(i)] \setcounter{enumi}{2}
		\item \label{itm:approx-abs-degree} $\deg(w,U) \geq (1 - 1/r^2)p|U|$ for every $w \in W$. 
	\end{enumerate}
	Then, for any $r$-edge-colouring of $G$, there is a collection $\cC$ of at most $3r^2$ monochromatic disjoint cycles, which together cover all but at most $K/p$ vertices of $W$ and $V(\cC) \subseteq U \cup W$.
\end{lemma}

Note that Lemma~\ref{lem:KMN+17} is a slight strengthening of Lemma~2.1 in~\cite{KMN+17}, which was originally stated for random graphs~$\G\sim\Gnp$. 
The proof of Lemma~2.1 in~\cite{KMN+17} only relies on the fact that w.h.p. $\G\sim\Gnp$ satisfies conditions \ref{itm:approx-abs-edges}-\ref{itm:approx-abs-degree}.
Thus, we omit its proof.

\begin{lemma}\label{lem:fine-absorption}
	Let $r ,t\geq 2$ be integers and $G$ be a graph with disjoint vertex sets $U,W \subseteq V(G)$ with $|W| \le t$.
	Suppose that
	\begin{enumerate}[\upshape(i)]
		\item \label{itm:fine-abs-W-U-degreee} for each $r$-set $S \subseteq W$, we have $\deg^*_G(S,U) \geq 6r^{r+1}t$; \nc{}{and}
		\item \label{itm:fine-abs-XYdensity} any disjoint subsets $X,Y \subseteq U$ with $|X|,|Y| \geq t$ satisfy  $e_G(X,Y) > 0$.
	\end{enumerate}
	Then, for any $r$-edge-colouring of~$G$, there is a collection $\cC$ of at most $400 r^4 \log r$ monochromatic disjoint cycles such that $W \subseteq V(\cC) \subseteq U \cup W$ and {$| V(\cC)| \leq 3|W|$.}
\end{lemma}

\begin{proof}
	Consider any $r$-edge-colouring of~$G$ with colours~$[r]$.
	Define an $r$-edge-coloured auxiliary (multi)graph~$H$ on~$W$ as follows.
	For each $j \in [r]$, we add an edge of colour $j$ between vertices $v,v' \in W$, if one of the following holds.
	\begin{enumerate}[(a)]
		\item \label{itm:fine-absneighbours-in-U} There are at least $2t$ vertices $u \in U$ such that $v u v'$ is a path of colour $j$.
		\item \label{itm:fine-abs-X-Y-matching} There is a matching $M$ of size $2t$ in $U$  such that $ v u u' v'$ (or $vu'uv'$) is a path of colour~$j$ for each $uu'\in M$.
	\end{enumerate}

	We claim that the independence number $\alpha(H)$ of $H$ is bounded by $2r-1$.
	Let $S$ be any $2r$-set of $W$.
	Partition $S$ into two $r$-sets $S_1$ and~$S_2$ with $S_i = \{v_{i,j}\colon j \in [r] \}$ for $i \in [2]$. 
	By~\ref{itm:fine-abs-W-U-degreee}, $N^*(S_i,U)$ has size at least $6r^{r+1}t$.
	By averaging, there exists $A_i \subseteq N^*(S_i,U)$ and colours $c_{i,j}$ for $j \in [r]$ such that $|A_i| \geq 6rt$ and $v_{i,j}u$ has colour $c_{i,j}$ for all $u \in A_i$ and $j \in [r]$. 
	If $c_{i,j} = c_{i,j'}$ for distinct $j,j' \in [r]$, then $v_{i,j}v_{i,j'}$ is an edge in~$H$ by~\ref{itm:fine-absneighbours-in-U}.
	Thus, without loss of generality, we may assume that $c_{i,j} = j$ for $i \in [2]$ and $j \in [r]$. 
	Let $A'_1 \subseteq A_1$ and $A'_2 \subseteq A_2$ be disjoint\nc{}{,} each of size~$3rt$.
	Let $M$ be a matching of maximum size between $A_1'$ and~$A_2'$ in~$G$.
	By~\ref{itm:fine-abs-XYdensity}, $M$ has size at least~$2rt$ and thus contains~$2t$ edges of some colour~$j \in [r]$.
	Then \ref{itm:fine-abs-X-Y-matching} yields that $v_{1,j} v_{2,j}$ is an edge in~$H$.
	Hence $\alpha(H) \le 2r-1$ as claimed.
	
	\tref{thm:sarkozy} implies that $H$ can be partitioned into a collection $\{F_i\}_{i \in [s]}$ of $s \leq 25(2r^2-r)^2 \log( 2r^2-r)\leq 400 r^4 \log r$ disjoint monochromatic cycles.
	\nc{}{By the definition of $H$ and as $t \geq |W|$, we can construct the desired collection of cycles covering $W$ greedily along the cycles $\{F_i\}_{i \in [s]}$.
		More precisely, suppose that $v_1v_2 \dots v_k v_1$ is a red cycle in $\{F_i\}_{i \in [s]}$. 
		Suppose we have already embedded vertices $v_1,\ldots,v_i$ on a red path $P_i$ that starts with $v_1$ and ends with $v_i$ such that $|V(P_i)| \leq 3i$.
		We extend this path to $v_{i+1}$ (with index modulo $k$) depending on whether case~\ref{itm:fine-absneighbours-in-U} or~\ref{itm:fine-abs-X-Y-matching} holds for the (red) edge $v_iv_{i+1} \in H$.
		If~\ref{itm:fine-absneighbours-in-U} holds with vertex $u$, then let $P_{i+1} = P_iuv_{i+1}$.
		If, on the other hand,~\ref{itm:fine-abs-X-Y-matching} holds with $u$ and  $u'$,  then let $P_{i+1} = P_iuu'v_{i+1}$.
		Note that this is possible as $t \geq|W| \geq k$.
		Using the same construction for each of the $F_i$'s, we are able obtain mutually disjoint cycles $\cC=\{C_i\}_{i \in [s]}$ satisfying $| V(\cC)| \leq 3|W|$.}
\end{proof}

Now we are ready to prove Lemma~\ref{lem:absorption}.
\begin{proof}[Proof of Lemma~\ref{lem:absorption}]
	Let $K= 24r^9/ \beta$.
	We claim that w.h.p. $\G \sim \Gnp$ satisfies the following.
	\begin{enumerate}[(a)]
		\item \label{itm:absorption-expected-neighbours} 
		For each $k \in \{1,r\}$, every $k$-set $S \subseteq V(\G)$ satisfies $\deg^*(S) \geq (1 - 1/r^4) p^{k} n $.
		\item \label{itm:absorption-edges}
		For two disjoint vertex sets $X,Y \subseteq V(\G)$, satisfying either of
		\begin{itemize}
			\item $|X|,|Y| \geq 144 \log n/p$, or
			\item $|X| \geq 768r/(\beta p)$ and $|Y|\geq  \beta n/(8r)$,
		\end{itemize}
		we have  $e(X,Y) = (1 \pm 1/4)p|X||Y|$.
		\item \label{itm:absorption-triples} For every family~$\cL$ of $\ell$ disjoint pairs of vertices and every set~$Y$ of~$3\ell$ vertices that is disjoint from each pair in~$\cL$, we have
		\begin{align*}
		\sum_{\{v,w\} \in \cL} \deg^*(\{v,w\}, Y) \le \begin{cases}
		72 \ell \log n &\text{if } \ell \le 6 \log n / p^2, \\
		6\ell^2 p^2 &\text{otherwise.}
		\end{cases}
		\end{align*}
		\item \label{itm:absorption-neighbours-in-U}
		For any vertex set $X$ of size $|X| \geq \beta n$, there is a set $Y \subseteq \overline{X}$ of size at most $K/p^{r}$ such that, for each $k \in \{1,r\}$, each $k$-set $S \subseteq \overline{X\cup Y}$ satisfies $\deg^*(S,X) =(1 \pm 1/r^4) p^k|X|$.
	\end{enumerate}
	Here, \ref{itm:absorption-expected-neighbours} is a \nc{straight forward}{straightforward} consequence of Lemma~\ref{lem:che}; \ref{itm:absorption-edges} follows from Lemma~\ref{lem:densityXY} with $1/4, \beta/(8r)$ playing the roles of $\alpha,\beta$;~\ref{itm:absorption-triples} follows from Lemma~\ref{lem:density_triples}; and~\ref{itm:absorption-neighbours-in-U} follows \nc{}{by choice of $K$} from two applications of Lemma~\ref{lem:dense-neighbours} with $k \in \{ 1,r \}$ and $\alpha=1/r^4$.
	We will deduce the lemma from these properties.
	
	Consider any $r$-edge-colouring of~$\G$ and any disjoint subsets of vertices $U$ and $W$  with $|U| \geq \beta n$ and $|W| \leq (\beta/400r)^{4} n$.
	By~\ref{itm:absorption-neighbours-in-U}, there is a set \nc{$W_1 \subseteq W$ with $|W_1| \leq K/ p^{r}$ such that, for each $k \in \{1,r\}$, each $k$-set~$ S \subseteq W \setminus W_1$ satisfies}{$Y \subseteq \overline{U}$ such that $W_1 \coloneqq W \cap Y$ has the following properties.
		It holds that $|W_1| \leq K/ p^{r}$ and, for each $k \in \{1,r\}$, every $k$-set~$ S \subseteq W \setminus W_1$ satisfies}
	\begin{align} \label{equ:degree-V-into-U}
	\deg^*(S,U) \ge (1 - 1/r^4) p^{k} |U|.
	\end{align}
	
	We first cover $W_1$.
	Let $t_1 = K/p^r$.
	By the choices of $K,r$ and $p$, we have
	\begin{align*}
	p^rn \geq 100 r^{r+1}t_1.
	\end{align*}
	As $|W_1| \leq K/ p^{r} = t_1$, it follows that, for any $r$-set $S \subseteq W_1$, we have
	\begin{align*}
	\deg^*(S,\c{W_1})
	\ge  \deg^*(S) - |W_1|
	&\overset{\ref{itm:absorption-expected-neighbours}}{\geq} (1-1/r^4)p^rn - t_1 \geq 6 r^{r+1} t_1.
	\end{align*}
	Note that~\ref{itm:absorption-edges} implies $e(X,Y) >0$ for all disjoint $X,Y \subseteq V(\G)$ with $|X|,|Y| \geq t_1$.
	By Lemma~\ref{lem:fine-absorption} with $\c{W_1},W_1,t_1$ playing the roles of $U,W,t$, there exists a collection~$\cC_1$ of at most ${400} r^4 \log r$ monochromatic disjoint cycles, such that $W_1 \subseteq V(\cC_1)$ and {$|V(\cC_1)| \leq 3 |W_1|$.}
	
	Let $W_2 = W \sm  V(\cC_1)$ and $U' =  U\sm V(\cC_1)$.
	Note that 
	\begin{align}
	|U \setminus U'| = |V(\cC_1) \cap U| \leq 2|W_1|\leq 2K/p^{r} \le  p^{r} |U| / r^5
		\label{eqn:U-U'} 
	\end{align}
	as the definition of~$p$ implies that $p^{2r} |U|/r^5 \ge 2^{6r} r^{10r}/\beta \ge 2K$.
	We pick a random partition $\{U_2,U_3\}$ of~$U'$ by assigning each $u \in U'$ independently at random to $U_2$ with probability~$1/2$ and to $U_3$ otherwise.
	A standard application of Lemma~\ref{lem:che} shows that w.h.p. the following holds for $i =2,3$:
	\begin{enumerate}[label={{(c$_{\arabic*}$)}}]
		\item \label{itm:U_i-geq-betan/4} $|U_i| \geq (1-1/r^4)|U'|/2 \ge \beta n /4$; and
		\item \label{itm:c_2} 	$\deg^*(S,U_i) \geq (1-1/r^4)\deg^*(S,U')/2$ for each $k \in \{1,r\}$, every $k$-set~$ S \subseteq W \setminus W_1$.
	\end{enumerate}
	We fix a partition  $\{U_2,U_3\}$ of~$U'$ with these properties.
	For each $k \in \{1,r\}$, every $k$-set~$ S \subseteq W \setminus W_1$ satisfies
	\begin{align}
	\nonumber 
		\deg^*(S,U_i) & 
		\overset{\mathclap{\text{\ref{itm:c_2}}}}{\ge} (1-1/r^4)\deg^*(S,U')/2 
		\ge (1-1/r^4)  ( \deg^*(S,U)- 	|U \setminus U'|) /2		  
		\\
		& \overset{\mathclap{\text{\eqref{equ:degree-V-into-U}, \eqref{eqn:U-U'}}}}{\ge} (1-2/r^4 - 1/r^5) p^{k} |U| /2	
		 \overset{\mathclap{\text{\ref{itm:U_i-geq-betan/4}}}}{\ge} (1 - 1/r^2) p^k |U_i|.	
	\label{eqn:equ:S-deg-W2-to-Ui} 	
	\end{align}

	Now we cover the majority of the vertices of~$W_2$ using~$U_2$.
	We apply Lemma~\ref{lem:KMN+17} with $U_2,W_2,\beta/4$ playing the roles of $U,W,\beta$.
	(This is possible as the conditions~\ref{itm:approx-abs-edges}, \ref{itm:approx-abs-triples} and~\ref{itm:approx-abs-degree} of Lemma~\ref{lem:KMN+17} are satisfied by properties \ref{itm:absorption-edges}, \ref{itm:absorption-triples} and~\eqref{eqn:equ:S-deg-W2-to-Ui}.)
	Thus there is a collection $\cC_2$ of at most $3r^2$ monochromatic disjoint cycles, which together cover all but at most $16000r^{4}/(\beta p)$ vertices of $W_2$ and $V(\cC_2) \subseteq U_2 \cup W_2$.
	
	It remains to cover $W_3 = W_2 \sm V(\cC_2)$.
	Let $t_2 = 16000r^{4}/(\beta p)$.
	For any $r$-set $S \subseteq W_3$,
	\begin{align*}
	\deg^*(S,U_3)  \overset{\mathclap{\text{\eqref{eqn:equ:S-deg-W2-to-Ui}}}}{\geq}
	(1 - 1/r^2) p^r |U_3| \overset{\mathclap{\text{\ref{itm:U_i-geq-betan/4}}}}{\geq} p^r \beta n /8 \ge	
	6 r^{r+1} t_2,
	\end{align*}
	where the last inequality holds by our choice of~$p$.
	Thus, as above, by Lemma~\ref{lem:fine-absorption} with $ U_3,W_3,t_2$ playing the roles of $U,W,t$, there exists a collection $\cC_3$ of at most ${400} r^4 \log r$ monochromatic disjoint cycles such that $W_3 \subseteq V(\cC_3) \subseteq U_3 \cup W_3$.
	
	Hence there are at most ${400}r^4 \log r + 3r^2 + {400}r^4 \log r \le {900} r^4 \log r$ disjoint monochromatic cycles $\cC \coloneqq \cC_1 \cup \cC_2 \cup \cC_3$, which together cover $W$.
	Note that 
	\begin{align*}
	V(\cC) \sm (U \cup W) = V(\cC_1) \sm (U \cup W) = V(\cC_1) \setminus W_1
	\end{align*}
	and hence 
	\begin{align*}
	|V(\cC) \sm (U \cup W)| = |V(\cC_1) \setminus W_1|
	\overset{\mathclap{\text{\eqref{eqn:U-U'}}}}{\le} 2K/p^r =48r^9/( \beta p^r).
	\end{align*}
	This finishes the proof of Lemma~\ref{lem:absorption}.
\end{proof}

\section{Proof of Lemma~\ref{lem:approximate-partition}} \label{sec:6}
In this section we prove Lemma~\ref{lem:approximate-partition}, i.e. we find a collection of monochromatic cycles that robustly partitions most of the vertices.
Before we set up definitions and tools, let us outline the approach.
Consider an $r$-edge-coloured graph $\G \sim \Gnp$. 
We apply the sparse regularity lemma to obtain a balanced vertex partition $\cV = \{V_i\}_{i \in [t]}$ of $\G$ in which almost every pair $(V_i,V_j)$ is (sparse\nc{}{ly}) regular. 
This allows us to define an $r$-edge-coloured reduced graph $R$ on~$[t]$, that encodes which pairs are regular and dense in one (or more) of the colours.
An easy calculation shows that (by deleting some edges) we may assume $R$ to be the union of at most $4r^2$ monochromatic components with $\delta(R) \geq 2t/3$.
Thus $R$ contains a perfect matching~$R'$.
{Our aim is to find a monochromatic cycle $C_i$ for each monochromatic component~$R_i$ of~$R$, such that the edges of $C_i$ are contained in pairs $(V_j,V_{j'})$ with $jj' \in E(R_i)$.}
The union of these $C_i$'s will form a monochromatic cycle partition. 
We use the sparse blow-up lemma of Allen, B\"ottcher, H\'an, Kohayakawa and Person~\cite{ABH+16} (Lemma~\ref{lem:blow-up}) to find these cycles. 
As usual, there is a small set $W \subset V(\G)$ whose vertices do not behave well enough to be included this way.
For example, we would need to remove vertices so that $(V_i,V_j)$ is super-regular for all $ij \in E(R')$ for our application of Lemma~\ref{lem:blow-up}.
(In the grand scheme, the set $W$ will be covered by Lemma~\ref{lem:absorption}. So we can ignore it in the following.)
We define the set $U$ by selecting each vertex of  $V(\G) \sm W$ with probability $2^{-11}$.
It remains to show that for any small set $U^+ \subseteq V(\G)$ and $U' \subset U$, the graph $\G - (W \cup U^+ \cup U')$ admits a partition into roughly $4r^2$ monochromatic cycles.
By the choice of~$U$ and as $U^+$ is small, the remainder of $\cV$ is still regular and fairly well balanced for our application of Lemma~\ref{lem:blow-up}.

\subsection{Sparse regularity and the blow-up lemma}
Our exposition of sparse regularity and the blow-up lemma closely follows \nc{}{the one of Allen, B{\"o}ttcher, H{\`{a}}n, Kohayakawa and Person}~\cite{ABH+16}.
Let $G=(V,E)$ be a graph and~$A$ and~$B$ be disjoint subsets of~$V$. 
For $0<p<1$, we define the \emph{$p$-density} of the pair $(A,B)$ to be $d_{G,p}(A,B)=e_G(A,B)/(p|A||B|)$. 
The pair $(A,B)$ is \emph{$(\eps,d,p)$-lower-regular} (in~$G$) if we have $d_{G,p}(A',B')\ge d-\eps$ for all $A'\subseteq A$ with $|A'|\ge\eps|A|$ and $B'\subseteq B$ with $|B'|\ge\eps |B|$.
Similarly, we say that $(A,B)$ is \emph{$(\eps,d,p)$-fully-regular}, if  $d_{G,p}(A',B')= d\pm\eps$ for all $A'\subseteq A$ with $|A'|\ge\eps|A|$ and $B'\subseteq B$ with $|B'|\ge\eps |B|$.\footnote{
	Note that Allen et al.~\cite{ABH+16} use the name regular pair to denote what we call a lower-regular pair.
	On the other hand, the standard term for a fully-regular pair is just a regular pair.
} 
We define a pair to be a sparse super-regular pair if it is a sparse lower-regular
pair and satisfies a minimum degree condition.
\begin{definition}[Sparse super-regularity]
	A pair $(A,B)$ in $G\subset\G$ is called \emph{$(\eps,d,p)$-super-regular} (in~$G$) if it
	is $(\eps,d,p)$-lower-regular and, for every $u\in A$ and $v\in B$, we have
	\begin{align*}
	\deg_G(u,B) &>(d-\eps) \max\{p |B|, \deg_{\G}(u,B)/2\}; \\
	\nonumber \deg_G(v,A) &>(d-\eps) \max\{p |A|, \deg_{\G}(v,A)/2\}.
	\end{align*}
\end{definition}
For our purpose, $\G$ will often be $G(n,p)$.
We remark that the term $(d-\eps)p|B|$ is a natural lower bound in the above
minimum degree condition by the following fact, which
easily follows from the definition of regularity.
\begin{fact}\label{fac:super-regular}
	Let~$(A,B)$ be an $(\eps,d,p)$-lower-regular pair. 
	Then $\deg_G(a,B) < (d-\eps)p|B|$ for at most $\eps|A|$ vertices $a \in A$.
\end{fact}
The next lemma also follows from the definition of regular pairs.
\begin{lemma}[Slicing lemma]
	\label{lem:slice}
	Let $(A,B)$ be an $(\eps,d,p)$-lower-regular pair and $A'\subseteq A$, $B'\subset
	B$ be sets of sizes $|A'|\ge\alpha|A|$, $|B'|\ge\alpha|B|$. Then
	$(A',B')$ is $(\eps/\alpha,d,p)$-lower-regular.
\end{lemma}

We say that a graph $G$ with density $p$ is $(\eta,D)$-\emph{upper-uniform with density $p$}, if, for all disjoint sets $U$ and~$W$ with $|U|,|W| \ge \eta |V(G)|$, we have $e_G(U,W)\le D p|U||W|$.
We will use the following sparse regularity lemma.

\begin{lemma}[{Sparse regularity lemma~\cite[Lemma 6.3]{ABH+16}}]\label{lem:sparse-regularity}
	For any real $D,\eps>0$, integers $r$ and $t_0$, there exist $\eta= \eta_{\ref{lem:sparse-regularity}}(D,\eps,r,t_0)>0$ and $t_1=t_{\ref{lem:sparse-regularity}}(D,\eps,r,t_0)$ with the following property.
	Let $\G$ be an $r$-edge-coloured, $(\eta,D)$-upper-uniform graph with density $p$ on at least $t_0$ vertices.
	Then there is a partition $\cV= \{V_i\}_{i \in [t]}$ of $V(\G)$ with the following properties.
	\begin{enumerate}[(a)]
		\item $t_0\le t \le t_1$;
		\item $||V_i|-|V_j|| \leq 1$ for all $i,j \in [t]$;
		\item all but at  most $\eps t^2$ pairs $(V_i,V_j)$ are $(\eps,d,p)$-fully-regular in each of the $r$ colours for some possibly different $d$.
	\end{enumerate}
\end{lemma}

The sparse blow-up lemma has the (reasonable) requirement that neighbourhoods of vertices inherit regularity.

\begin{definition}[Regularity inheritance]
	Let $A$, $B$ and $C$ be vertex sets in
	$G\subset\G$, where~$A$ and~$B$ are disjoint and~$B$ and~$C$
	are disjoint, but we do allow $A=C$. We say that $(A,B,C)$ has
	\emph{one-sided $(\eps,d,p)$-inheritance} if for each $u\in A$, the pair
	$\big(N_\G(u,B),C\big)$ is $(\eps,d,p)$-lower-regular.
\end{definition}

{The next lemma states that there are only few vertices in $\Gnp$ which do not inherit regularity.}

\begin{lemma}[One-sided regularity inheritance in $\Gnp$,~{\cite[Lemma 1.26]{ABH+16}}]
	\label{lem:oneRI}
	For each $\eps',d >0$, there are $\eps_0 = \eps_{\ref{lem:oneRI}}(\eps',d) >0$ and
	$C = C_{\ref{lem:oneRI}}(\eps',d)$ such that, for all $0<\eps<\eps_0$ and $0<p<1$, w.h.p.
	$\G \sim \Gnp$ has the following property. 
	
	Let $G\subset\G$ be a graph and $X,Y$ be disjoint subsets of
	$V(\G)$. 
	If $(X,Y)$ is $(\eps,d,p)$-lower-regular in~$G$ and
	\[|X|\ge C\max\big(p^{-2},p^{-1}\log n\big) \quad\text{and}\quad|Y|\ge Cp^{-1}\log (en/|X|),\] 
	then there are at most $Cp^{-1}\log (en/|X|)$ vertices $z\in V(\G)$ such that 
	the pair $\big(N_\G(z,X),Y\big)$ is not $(\eps',d,p)$-lower-regular in~$G$.
\end{lemma}

The setting in which the blow-up lemma works is as follows. 
Let $G$ and $H$ be two graphs on the same number of vertices.
Let $\cV=\{V_i\}_{i\in[t]}$ and $\cX=\{X_i\}_{i\in[t]}$ be partitions of $V(G)$ and $V(H)$, respectively. 
We call the parts~$V_i$ of~$G$ \emph{clusters}.
We say that~$\cV$ and~$\cX$ are \emph{size-compatible} if $|V_i|=|X_i|$ for all
$i\in[t]$. 
Moreover, for $\kappa\geq 1$, we say that $\cV$ is
\emph{$\kappa$-balanced} if there exists $m\in\mathbb{N}$ such that we have
$m\leq |V_i|\leq \kappa m$ for all $i \in[t]$.
Our goal is to embed~$H$ into~$G$ respecting these partitions.

We will have two reduced graphs~$R$ and~$R'\subseteq R$,
where~$R'$ represents super-regular pairs and~$R$ regular pairs. More
precisely, we require the following properties of~$R$ and~$R'$ and the
partitions~$\cV$ and~$\cX$ of~$G$ and~$H$.

\begin{definition}[Reduced graphs and one-sided inheritance]
	\label{def:RGHpartition}
	Let~$R$ and~$R'$ be graphs on vertex set $[t]$.
	\begin{itemize}
		\item $(H,\cX)$ is an \emph{$R$-partition} if $H[X]$ is empty for all $X \in \cX$ and $ij \in E(R)$ whenever $e_H(X_i,X_j) >0$.
		\item $(G,\cV)$ is an \emph{$(\eps,d,p)$-regular $R$-partition} if $(G,\cV)$ is an $R$-partition and the pair $(V_i,V_j)$ is $(\eps,d,p)$-lower-regular for all $ij\in E(R)$.
	\end{itemize}
	In this case we also say that~$R$ is a \emph{reduced graph} of the partition~$\cV$.
	\begin{itemize}
		\item $(G,\cV)$ is \emph{$(\eps,d,p)$-super-regular on $R'$} if the pair $(V_i,V_j)$ is $(\eps,d,p)$-super-regular for all $ij\in E(\nc{R}{R'})$.
	\end{itemize}
	Suppose now that $(G,\cV)$ is an $(\eps,d,p)$-regular $R'$-partition.
	\begin{itemize}
		\item $(G,\cV)$ has \emph{one-sided inheritance} on $R'$ 
		if $(V_i,V_j,V_k)$ has one-sided $(\eps,d,p)$-inheritance
		for all $ij, jk\in E(R')$.
	\end{itemize}
\end{definition}

Next we define define the so-called ``buffer sets'' of vertices in~$H$.
\nc{}{The purpose of these buffer sets in the context of the blow-up lemma is to ensure that $H$, the graph which we intend to embed, is sufficiently dense in the pairs corresponding to~$R'$.}
Note that we restrict ourselves to the case when $H$ is triangle-free.
\begin{definition}[Buffer sets]\label{def:buffer-set}
	Let $H$ be a triangle-free graph with vertex partition $\cX=\{X_i\}_{i\in[t]}$.
	Let $R'\subseteq R$ be graphs on vertex set $[t]$.
	Suppose that $(H,\cX)$ is an $R$-partition.
	We say the family $\tcX=\{\tX_i\}_{i\in[t]}$ of subsets $\tX_i\subseteq X_i$ is an \emph{$(\alpha,R')$-buffer} for $H$ if
	\begin{itemize}
		\item $|\tX_i|\ge\alpha |X_i|$ for all $i\in[t]$  and 
		\item for each $i\in[t]$ and each $x\in\tX_i$, the first and second
		neighbourhoods of $x$ \emph{go along $R'$}, that is,
		for each $xy,yz\in E(H)$ with $y\in X_j$ and $z\in X_k$, we have $ij, jk\in E(R')$.
	\end{itemize}
\end{definition}

Now we are ready to state the blow-up lemma.
{Note that the general version of Lemma~\ref{lem:blow-up} applies also to graphs $H$ that contain triangles.}
\begin{lemma}[{Blow-up lemma for $\Gnp$~\cite[Lemma 1.21]{ABH+16}}]\label{lem:blow-up}
	For all $t_1, \Delta\ge 2$, $\Delta_{R'}$, $\alpha, d>0$, $\kappa>1$
	there exist $\eps = \eps_{\ref{lem:blow-up}}(\Delta, \Delta_{R'}, \alpha, d, \kappa )>0$ and a constant $C = C_{\ref{lem:blow-up}}(\Delta, \Delta_{R'}, \alpha, d, \kappa , t_1)$ such that, for $p>C({\log n}/{n})^{1/\Delta}$, w.h.p. the random graph $\G \sim \Gnp$ 	satisfies the following.
	
	Let $R$ be a graph on $t\le t_1$ vertices and let $R'\subseteq R$ be a spanning
	subgraph with $\Delta(R')\leq \Delta_{R'}$.
	Let $H$ and $G\subseteq \G$ be graphs with $\kappa$-balanced size-compatible vertex partitions $\cX=\{X_i\}_{i\in[t]}$ and $\cV=\{V_i\}_{i\in[t]}$, respectively, which have
	parts of size at least $m\ge n/(\kappa t_1)$.
	Let $\tcX=\{\tX_i\}_{i\in[t]}$ be a family of subsets of $V(H)$.
	Suppose that
	\begin{enumerate}[\upshape(i)]
		\item \label{itm:blow-up-triangle-free} $H$ is triangle-free;
		\item \label{itm:blow-up-buffer} $\Delta(H)\leq \Delta$, $(H,\cX)$ is an $R$-partition,  $\tcX$ is an $(\alpha,R')$-buffer for $H$; and 
		\item \label{itm:blow-up-reduced-graph} $(G,\cV)$ is an $(\eps,d,p)$-regular $R$-partition, which
		is $(\eps,d,p)$-super-regular on~$R'$ and
		has one-sided $(\eps,d,p)$-inheritance on~$R'$.
	\end{enumerate}
	Then there is a graph embedding $\psi\colon V(H) \to V(G)$ such that $\psi(X_i) =  V_i$ for each $i \in [t]$ and $\psi (v) \psi(v') \in E(G)$ for each $vv' \in E(H)$.
\end{lemma}

\subsection{Preliminaries}
The next lemma allows us to assume that the reduced graph consists of a bounded number of monochromatic components by reducing its minimum degree.

\begin{lemma}\label{lem:choose-components}
	Let $\delta \ge \gamma> 0 $ and $G$ \nc{}{be} an $r$-edge-colour\nc{ing of}{ed} graph on $t$ vertices with $\delta(\nc{R}{G}) \geq \delta t$.
	Then $G$ contains a spanning subgraph~$\nc{H}{R}$ with $\delta(R) \geq (\delta- \gamma) t$, which is the union of at most $r^2/\gamma$ monochromatic components.
\end{lemma}
\begin{proof}
	Define $R$ to be the union of all monochromatic components of~$G$ of size at least $\gamma t/r$.
	Note that, for each colour, there are at most $r/\gamma$ such components.
	Hence $R$ consists of at most $r^2/\gamma$ monochromatic components.
	Since each vertex~$v$ is in at most $r$ monochromatic components and each component not in~$R$ has at most $\gamma t/r$ edges touching~$v$, we have $\deg_{R}(v) \ge \delta(G) - r \cdot \gamma t/r$.
	In particular, $R$ spans~$G$.
\end{proof}

Let $R'$ be a perfect matching in~$R$.
Suppose that $(\G, \cV)$ is a regular $R$-partition and a super regular $R'$-partition.
Recall that we aim to find a small monochromatic cycle partition of~$\G$ using Lemma~\ref{lem:blow-up}.
The following lemma will enable us to define the graph~$H$ needed in the hypothesis of Lemma~\ref{lem:blow-up}.
Since $R'$ is a (perfect) matching and $H$ consists of disjoint cycles, verifying the existence of an $(\alpha,R')$-buffer for~$H$ reduces to showing that, for each $ij \in E(R')$, there is a long path in $H[X_i,X_j]$.

\begin{lemma}\label{lem:cycle-allocation}
	Let $s,t,m \in \mathbb{N}$ with $m \geq 90t^3s$ and $t$ even.  
	Let $R$ be a graph on $[t]$ with $\delta(R) \geq 2t/3$ and $R' \subseteq R$ be a perfect matching.
	Suppose that $R$ is the union of edge-disjoint connected subgraphs $R_1,\ldots,R_s$.
	Let $\{x_i\}_{i \in [t]}$ be a family of integers satisfying $m \leq x_i \leq 10m/9$.
	Then there is a graph $H$, a partition $\cX =\{X_i\}_{i \in [t]}$ of $V(H)$ and a family $\tilde\cX = \{\tilde X_i \}_{i \in [t]}$ of subsets of $V(H)$ with the following properties.
	\begin{enumerate}[label={\rm (\roman*)}]
		\item \label{itm:cycle-allocation1}$|X_i| = x_i$ for each $i \in [t]$;
		\item \label{itm:cycle-allocation2}$H$ is triangle-free;
		\item \label{itm:cycle-allocation3} $\tilde\cX$ is a\nc{n}{} $(1/50,R')$-buffer for $H$; and 
		\item \label{itm:cycle-allocation4}
		$H$ is the union of vertex-disjoint cycles $C_1,\ldots,C_s$ and at most one isolated vertex such that $(C_k,\cX)$ is an $R_k$-partition for each $k \in [s]$. 
		In particular, $(H,\cX)$ is an $R$-partition and $\Delta(H) \le 2$. 
	\end{enumerate}
\end{lemma}

\begin{proof}
	Let $\cX =\{X_i\}_{i \in [t]}$ be a family of disjoint vertex sets satisfying $|X_i| = x_i$.
	For each $k \in [s]$, let $G_k$ be the graph on $\bigcup_{i \in V(\nc{G_k}{R_k})} X_i$ which contains an edge $vw$ precisely when $v \in X_i$ and $w \in X_j$ for some edge $ij \in E(R_k)$.
	\nc{}{So $G_k$ is complete between $X_i$ and $X_j$ for every edge $ij$ of $R_k$.}
	\nc{Thus}{It follows that} $G \coloneqq \bigcup_{k \in [s]} G_k$ is a union of edge-disjoint graphs and $(G,\cX)$ is an $R$-partition.
	
	First, we claim that there exists \nc{}{a collection of} disjoint cycles $\cC'=\{C'_1, \dots, C'_s\}$ in~$G$ such that, for each $k \in [s]$, 
	\begin{itemize}
		\item $C_k' \subseteq G_k$ and $C_k'$ is triangle-free;
		\item for each edge $ij \in E(R_k)$, there is a path $P_{ij} \subseteq C'_k$ which alternates between $X_i$ and $X_j$ and has order 
		\begin{align*}
		|V(P_{ij})| \ge \begin{cases}
		4m/45  \ge   ( |X_i|+ |X_j| ) / 25+4& \text{if $ij \in E(R')$,} \\
		4 & \text{otherwise;}
		\end{cases}
		\end{align*}		
		\item $| V(\cC')| \le 2 t^3 s + 2 t m/45 \le tm /15$.
	\end{itemize}
	To see that such cycles exist, note that, by the connectivity of $R_k$, any two vertices of $G_k$ are connected by a path of order at most $t$.
	Thus at most $t^3$ vertices are needed for a cycle $C_k'$ in $G_k$ to ``visit'' all edges of $R_k$, that is, \nc{}{for} the cycle $C_k'$ \nc{contains}{to contain} an edge between $X_i$ and $X_j$ for each $ij \in E(R_k)$.	
	Moreover, by replacing each edge of $C_k'$ with a path of length~3 (between the same clusters), we can guarantee triangle-freeness.
	Thus for each $k \in [s]$ and $ij \in E(R_k)$, $C_k'$ contains a path $P_{ij}$ of length~$3$ alternating between $X_i$ and~$X_j$ and $| V(\cC')| \le 2 t^3 s$.
	Since $|X_i| \geq m \geq 2t^3s+2m/45$ and $R'$ is a perfect matching, we further extend\nc{s}{} each $P_{ij}$ with $ij \in E(R')$ to have length at least $4m /45$.
	Thus our claim\nc{s}{} holds. 
	
	Let $G' = G - \bigcup_{k \in [s]} V(C'_k)$, so $|V(G')| \geq 14tm/15$.
	Recall that $\delta(R) \ge 2t/3$ and $|X_i| \le 10 m /9$. 
	Hence 
	\begin{align*}
			\delta(G') & \geq |V(G')| - t/3 \cdot 10m/9 \ge |V(G')|/2.
	\end{align*}
	By Dirac's theorem, $G'$ contains a matching~$M$, which misses at most 1 vertex. 
	For each $k \in  [s]$, we obtain a cycle~$C_k$ from $C'_k$ by extending each path $P_{ij}$ (with $ij \in E(R_k)$) to include all edges of~$M$ between $X_i$ and~$X_j$.
	To be precise, if the edges in~$M$ between $X_i$ and $X_j$ are $x_{i,1} x_{j,1},~ x_{i,2} x_{j,2},~\dots,~x_{i,p} x_{j,p}$ with $x_{i,q} \in X_i$ and $x_{j,q} \in X_j$, then we attach the path $x_{i,1} x_{j,1} x_{i,2} x_{j,2}\dots x_{i,p} x_{j,p}$ to either the start or end of~$P_{ij}$ (and define $C_k$ accordingly).

	Let $H$ be the graph on $V(G)$ with $E(H) =  \nc{E(}{}\bigcup_{k \in [s]} \nc{}{E(}C_k)$.
	By our construction, \ref{itm:cycle-allocation1}, \ref{itm:cycle-allocation2} and \ref{itm:cycle-allocation4} hold.
	Recall \nc{}{that} for each $ij \in E(R')$, there exists a \nc{}{path} $P_{ij}$ in~$H$ alternating between $X_i$ and $X_j$ with $|V(P_{ij})| \geq (|X_i|+|X_j|)/{25}+4$.
	Let $P'_{ij}$ be the path obtained from $P_{ij}$, by deleting the two vertices on each end of~$P_{ij}$.
	It follows that  ${|V(P'_{ij})| \geq (|X_i|+|X_j|)/{25}}$.
	Moreover, the first and second neighbourhoods of each vertex $v \in V(P'_{ij}) \cap X_i$ are contained in $X_i$ and $X_j$, respectively, and vice versa.
	Hence we can choose a {$(1/50,R')$-buffer} $\tilde\cX = \{\tilde X_i \}_{i \in [t]}$ for~$H$, namely $\tilde{X}_i = V(P'_{ij}) \cap X_i$ for each $ij$ in~$R'$.
	\end{proof}

\subsection{Proof of Lemma~\ref{lem:approximate-partition}}

\nc{}{This subsection is dedicated to the proof of Lemma~\ref{lem:approximate-partition}.}
\nc{}{We begin by setting up the following constants.}
Let
\begin{align*}
	\eps_{\ref{lem:blow-up}} & = \min\{ 2^{-8}, \eps_{\ref{lem:blow-up}}(2, 1, 1/50, 1/{r}, 10/{9} ) \}, &
	\eps_{\ref{lem:oneRI}} & =  \eps_{\ref{lem:oneRI}}(\eps_{\ref{lem:blow-up}}/{2},1/r), \\
	C_{\ref{lem:oneRI}} & = C_{\ref{lem:oneRI}}({\eps_{\ref{lem:blow-up}}}/{2},1/r), &
	\eps &=  \min\{ 2^{-10}, \eps_1^{2}, {\eps_{\ref{lem:blow-up}}}, \eps_{\ref{lem:oneRI}} \} /4,\\
	t_0 &= 1/ \eps^2, &
	\eta &= \eta_{\ref{lem:sparse-regularity}}(2,\eps^2,r,t_0),\\
	t_1 &= \max \{ t_{\ref{lem:sparse-regularity}}(2,\eps^2,r, t_0 ) , 1/ (2 \eta)\},&
	C& = C_{\ref{lem:blow-up}}(2, 1, {1}/{50}, {1}/{r}, {10}/{9}  , t_1),
\end{align*}
where the functions $\eps_{\ref{lem:blow-up}}, C_{\ref{lem:blow-up}}$ are given by Lemma~\ref{lem:blow-up}; $\eps_{\ref{lem:oneRI}}, C_{\ref{lem:oneRI}}$ by Lemma~\ref{lem:oneRI}; $\eta_{\ref{lem:sparse-regularity}},t_{\ref{lem:sparse-regularity}}$ by Lemma~\ref{lem:sparse-regularity}.

\noindent\textbf{Properties of the random graph.}
Let $p=p(n) \geq  C(\log n /n)^{1/(r+1)}$.
Let $\G \sim \Gnp$.
Note that w.h.p. $\G$ satisfies the conclusions of Lemmas~\ref{lem:oneRI} and~\ref{lem:blow-up}. \nc{(as $\Delta = 2$)}{} 
\nc{}{(In particular the conditions on the each of the $p$'s hold as $r+1 \geq 2 = \Delta$.)}
By Lemma~\ref{lem:densityXY} with $\eps/2, 1/(2t_1)$ playing the roles of $\alpha, \beta$, w.h.p. $\G$ satisfies the following property 
\begin{align}\label{equ:XY-density}
\begin{minipage}[c]{0.8\textwidth}
$e_\G(X,Y)  = (1 \pm \eps/2) p |X| |Y|$ for all disjoint $X,Y\subseteq V(\G)$ with $|X| \geq 48t_1/(\eps^2p)$ and $|Y| \ge  n/(2t_1)$.
\end{minipage}\ignorespacesafterend 
\end{align}
We will derive the lemma from these properties.
\medskip

\noindent\textbf{Defining regular partition and reduced graph.}
Consider any $r$-edge-colouring of~$\G$ with colours~$[r]$.
By~\eqref{equ:XY-density}, any disjoint sets $X,Y \subseteq V(\G)$ of cardinality at least $\eta n \geq n/(2t_1)$ satisfy $e_\G(X,Y) \leq 2 p |X| |Y|$.
This implies that $\G$ is $(\eta,\nc{D}{2})$-upper-uniform with density $p$.
Therefore Lemma~\ref{lem:sparse-regularity} (with $\eps^2$ playing the role of $\eps$) guarantees a vertex partition $\cV= \{V_i\}_{i \in [t']}$ of $V(\G)$ such that:
\begin{enumerate}[label={(a$_{\arabic*}$)}]
	\item $t_0 \le t' \le t_1$;
	\item $||V_i|-|V_j|| \leq 1$ for all $i,j \in [t']$;
	\item \label{itm:regular-pairs} all but at  most $\eps^2 t'^2$ pairs $(V_i,V_j)$ are $(\eps^2,d,p)${-fully-regular} in each of the $r$ colours for some possibly different $d$.
\end{enumerate}

We define a reduced graph~$T'$ on~$[t']$ such that $ij$ is an edge in~$T'$ if and only if $(V_i,V_j)$ are $(\eps^2,d,p)${-fully-regular} in each of the $r$ colours for some possibly different~$d$. 
So $T'$ contains all but $\eps^2 t'^2$ edges by~\ref{itm:regular-pairs}. 
By deleting all vertices with degree less than $(1-\eps)t'$ (and possibly one vertex more), we obtain a subgraph~$T$ of~$T'$ on some even $t \ge (1-2\eps)t'$ vertices with $\delta(T') \geq (1 - 3\eps)t$.
Without loss of generality, we can assume that $V(T) = [t]$.

For a colour $c \in [r]$, we denote by $\G_{c}$ the \nc{spanning }{}subgraph of $\G$, which contains all edges of colour $c$.
We now edge-colour~$T$ as follows. 
Consider an edge~$ij$ of~$T$. 
By~\eqref{equ:XY-density}, $e_\G(V_i,V_j) \geq (1-\eps/2)p|V_i||V_j|$.
Hence there exists some colour $c \in [r]$ such that $e_{\G_c}(V_i,V_j) \geq (1/r-\eps/2)p|V_i||V_j|$.
Moreover, $(V_i,V_j)$ is an $(\eps,1/r,p)${-lower-regular} pair in $\G_c$ as $\eps \geq \eps/2 +\eps^2$ and $ij \in E(T)$.
We colour the edge $ij$ (of~$T$) with such a colour~$c$ (if there is more than one such colour, then we choose one arbitrarily). 

We apply Lemma~\ref{lem:choose-components} with $\gamma =1/4$ to obtain a spanning subgraph~$R$ of~$T$ with $\delta(R) \geq 2 t/3$ and which is the union of at most $4r^2$ monochromatic components.
Since $t$ is even, $R$ contains a perfect matching~$R'$.

Define $G'$ to be the spanning subgraph of $\G[\bigcup_{i \in [t]} V_i]$ obtained by  keeping an edge $vw$ precisely when there is an edge $ij$ in~$R$ such that $v \in V_i$, $w\in V_j$ and $vw$ has the same colour as $ij$.
\medskip

\noindent\textbf{Defining $W$.}
The set $W$ will consist of four vertex sets $\bigcup_{i \notin [t]} V_i$, $\Vexc$, $\Vdeg$ and $\Vinh$.
 Each of these four sets presents an obstacle to the application of Lemma~\ref{lem:blow-up}.
	The vertices $i \notin [t]$ are not covered by $R$ and hence the set $\bigcup_{i \notin [t]} V_i$ lies outside of the scope of Lemma~\ref{lem:blow-up}.
	The set $\Vexc$ will contain the vertices which do not behave typically with respect to the regularity of the pairs of $R'$.
	The set $\Vdeg$ will contain the vertices which do not behave typically with respect to the random graph $\G$ and the partition $\cV$.
The removal of $\Vexc \cup \Vdeg$ \nc{will ensure}{then ensures} that $(G, \cV)$ is super-regular on~$R'$.
Finally, the set~$\Vinh$ contains the vertices that do not inherit regularity with respect to~$R'$.
We now formally define $\Vexc$, $\Vdeg$ and $\Vinh$.

First we define~$\Vexc$.
For each $i \in [t]$ with $ij \in E(R')$, \nc{define}{set}
\begin{align*}
	\Vexc_{i} & = \{v \in V_i \colon\  \deg_{G'}(v,V_j)  < (1/r-\eps) p |V_j| \},
\end{align*}
that is, $\Vexc_{i} $ consists of all vertices in $V_i$ that have less than the expected number of neighbours in~$V_j$ in~$G'$.
Note that 
\begin{align*}
|\Vexc_i| \le \eps |V_i|
\end{align*}
by Fact~\ref{fac:super-regular}.
Let $\Vexc = \bigcup_{i \in [t]} \Vexc_i$ and $V'_i = V_i \setminus \Vexc_i$ for all $i \in [t]$. 

Next, we define $\Vdeg$.
For each $i \in [t]$, let $\Vdeg_i$ be the set of vertices $v \in V(\G) \sm V'_i$ with $\deg_\G(v,V_i') \neq (1 \pm \eps)p|V_i'|$ and set $\Vdeg  = \bigcup_{i\in [t]} \Vdeg_i$.
\nc{}{So~\eqref{equ:XY-density} together with the fact that $|V_i'| \geq n/(2t_1)$, gives $|\Vdeg_i| \leq 48t_1/(\eps^2p)$.}
Therefore we can bound
\begin{align*}
|\Vdeg | \le 48 t_1^2/(\eps^2 p).
\end{align*}

Consider any $i \in [t]$, and let $ij \in E(R')$.
Lemma~\ref{lem:slice} implies that $(V'_i,V'_j)$ is $(2\eps,1/r,p)$-lower-regular in $G'$.
Since $n$ is large, we have 
\begin{align*}
|V_i'|\ge C_{\ref{lem:oneRI}}\max\big(p^{-2},p^{-1}\log n\big) \quad\text{and}\quad|V_j'|\ge C_{\ref{lem:oneRI}}p^{-1}\log (en/|V_i'|).
\end{align*}
Let $\Vinh_{i}$ be the set of vertices $z\in V(\G)$ such that the pair $\big(N_\G(z,V_i'),V_j'\big)$ is not $( \eps_{\ref{lem:blow-up}}/2 , 1/r , p )$-lower-regular in $G'$ and set
$\Vinh = \bigcup_{i \in [t] }\Vinh_{i}$.
\nc{Since $2 \eps < \eps_{\ref{lem:oneRI}}$, Lemma~\ref{lem:oneRI} implies that $|\Vinh_i| \le  C_{\ref{lem:oneRI}} p^{-1}\log(2et') $ for each $i \in [t]$ and so }{Recall our choice of $\eps_{\ref{lem:oneRI}}$, $C_{\ref{lem:oneRI}}$ and note that $2 \eps < \eps_{\ref{lem:oneRI}}$.
We apply Lemma~\ref{lem:oneRI} with $2\eps$, $\eps_{\ref{lem:blow-up}}/{2}$ and $1/r$ playing the roles of $\eps$, $\eps'$ and $d$, respectively.
This yields that $|\Vinh_i| \le  C_{\ref{lem:oneRI}} p^{-1}\log(2et') $ for each $i \in [t]$. It follows that}
\begin{align*}
	|\Vinh|\leq C_{\ref{lem:oneRI}} t p^{-1} \log( 2 et') .
\end{align*}

Finally, set $W = \bigcup_{i \notin [t]} V_i \cup \Vexc \cup \Vdeg \cup \Vinh$.
In summary, we have, for each $i \in [t]$ with $ij \in E(R')$,
\begin{enumerate}[label={(b$_{\arabic*}$)}]
	\item \label{itm:b1}
	$|W| \le \eps_1 n$ and $|W \cap V_i| \le 3 \eps |V_i| \le 6 \eps |V_i'|$; 
	\item \label{itm:b2}
	for all $v \in V_i'$, $ \deg_{G'} ( v ,V_j) \ge ( 1/r - \eps ) p | V_j| $;
	\item \label{itm:b3}
	for all $u  \notin W \cup V_i' $, $\deg_\G (u,V_i') = (1 \pm \eps)p|V_i'|$; \nc{}{and}
	\item \label{itm:b4}
	for all $z \notin W$, 
	the pair $\big(N_\G(z,V_i'),V_j'\big)$ is $( \eps_{\ref{lem:blow-up}}/2,1/r,p)$-lower-regular in~$G'$.
\end{enumerate}
Since $t \ge (1- \sqrt{\eps})t'$ and $n$ is large, we have $|\bigcup_{i \notin [t]} V_i| \le \sqrt{\eps} t' \lceil n/t' \rceil \le \eps_1 n /2$.
We deduce that \ref{itm:b1} holds.
Note that \ref{itm:b2}, \ref{itm:b3} and \ref{itm:b4} hold as $v \in V_i \setminus \Vexc_i$, $u \notin V'_i \cup \Vdeg $ and $z \notin \Vinh_i$, respectively.
\medskip

\noindent\textbf{Defining $U$.}
\nc{We will pick $U = \bigcup_{ i \in [t]} U_i$ such that for each $i \in [t]$,}{For $i \in [t]$, we pick a random set $U_i$ by selecting each $v \in V_i' \setminus W $ with probability $3 \cdot 2^{-12}$ (i.e. half way between $2^{-11}$ and $2^{-10}$).
A standard application of Lemma~\ref{lem:che} shows that w.h.p. $U \coloneqq \bigcup_{ i \in [t]} U_i$ has the following properties:}
\begin{enumerate}[label={(c$_{\arabic*}$)}]
	\item $U_i \subseteq V_i \setminus W$;
	\item \label{itm:U-sparse} $2^{-11} | V_i '| \le |U \cap V_i'| \leq 2^{-10} |V_i'|$; \nc{}{and}
	\item 
	for each $v \in V(\G) \sm ( W \cup V_i')$, we have 
	\begin{align*}
	\deg_{\G}( v , U  \cap V_i') \leq 2^{-10}\deg_{\G}(v,V_i') \nc{\overset{\text{\ref{itm:b3}}}{\leq} 2^{-9} p |V'_i|}{.}
	\end{align*}
\end{enumerate}
\nc{(Indeed, such $U_i$ exists by selecting each $ v \in V_i \setminus W $ with probability $3\cdot 2^{-12}$ and together with a standard application of Lemma~\ref{lem:che}.)}{Let us fix $U = \bigcup_{ i \in [t]} U_i$ with these properties.} 
Note that $2^{-12}n \le |U| \le 2^{-10} n$.
\nc{}{Moreover, \ref{itm:b3} yields that,
\begin{enumerate}[label={(c$_{\arabic*}'$)}]\setcounter{enumi}{2}
	\item \label{itm:deg-U} for each $v \in V(\G) \sm ( W \cup V_i')$, we have 
	\begin{align*}
	\deg_{\G}( v , U  \cap V_i') \leq  2^{-9} p |V'_i|.
	\end{align*}
\end{enumerate}}

\medskip

\noindent\textbf{Finding \nc{}{a} monochromatic cycle partition.}
We now verify that the lemma holds with our choices of $W$ and~$U$. 
Consider any $U' \subseteq U$ and $U^+ \subseteq V(\G)$ with 
\begin{align}\label{equ:size-W'}
|U^+| \leq  2^{18}r^8/p^r \leq 2^{-10} p |V_i'|.
\end{align}
(Here we used the fact that $|V_i| \geq n/t_1$ and $p \geq (\log n /n)^{1/{(r+1)}}$.)
Let $Q = W \cup U^+ \cup U'$.
To finish the proof, we will show that  $\G - Q$ admits a partition into at most $4r^2+1$ monochromatic cycles.

Let $G = G'-Q$ and define a partition $\cV^*=\{V_i^*\}_{i\in [t]}$ with $V_i^* = V_i \sm Q$ for $i \in [t]$.
By~\ref{itm:b1}, \eqref{equ:size-W'} and~\ref{itm:U-sparse},  we have
\begin{align}\label{equ:V*-V}
|V^*_i| &\geq   |V_i| - |W \cap V_i| - |U \cap V_i| - |U^+| \geq (1-2^{-8})|V_i|\\
\nonumber
 &\ge 10n/11t_1.
\end{align}
\nc{Hence $\cV^*$ is $(10/9)$-balanced.
	So there is some integer $m$ with $m \leq |V^*_i| \leq 10m/9$ for every $i \in [t]$.}{Hence there is some integer $m$ with $m \leq |V^*_i| \leq 10m/9$ for every $i \in [t]$.
	So $\cV^*$ is $(10/9)$-balanced.}

Denote the monochromatic components of~$R$ by $R_1, \ldots, R_{s}$ with $s \leq 4r^2$.
We apply Lemma~\ref{lem:cycle-allocation} with $|V^*_i|$ playing the role\nc{s}{} of $x_i$ \nc{and}{to} obtain a graph~$H$, a partition $\cX =\{X_i\}_{i \in [t]}$ of~$V(H)$ and a family~$\tilde\cX = \{\tilde X_i \}_{i \in [t]}$ of subsets of $V(H)$ with the following properties.
\begin{enumerate}[label={(d$_{\arabic*}$)}]
		\item \label{itm:d1} $\cX$ is size-compatible with $\cV^*$;
		\item \label{itm:d2} $H$ is triangle-free;
		\item \label{itm:d3} $\tilde\cX$ is a\nc{n}{} $(1/50,R')$-buffer for $H$; and 
		\item \label{itm:d4}
		$H$ is the union of vertex-disjoint cycles $C_1,\ldots,C_s$ with at most one isolated vertex such that $(C_k,\cX)$ is an $R_k$-partition for each $k \in [s]$. 
		In particular, $(H,\cX)$ is an $R$-partition and $\Delta(H) \le 2$. 
	\end{enumerate}
To find a monochromatic cycle partition of $G$ into at most $4r^2+1$ cycles, it suffices to show that there exists an embedding~$\psi$ of~$H$ into~$G$ with $\psi(X_i) = V_i^*$ for each $i \in [t]$.
Therefore, it suffices to show that we can apply Lemma~\ref{lem:blow-up} with $2,1,1/50,1/r, 10/9$ playing the roles of $\Delta, \Delta_{R'}, \alpha, d, \kappa$. 
\medskip

\noindent\textbf{Verifying the conditions of the blow-up lemma.}
Note that~\ref{itm:d2}--\ref{itm:d4} implies conditions~\ref{itm:blow-up-triangle-free} and~\ref{itm:blow-up-buffer} of Lemma~\ref{lem:blow-up}.
\nc{Each}{Moreover, each} $V_i^*$ has the desired size by~\eqref{equ:V*-V}.
It remains to show that condition~\ref{itm:blow-up-reduced-graph} of Lemma~\ref{lem:blow-up} is satisfied, which is covered by the following two claims.
Note that for all $i \in [t]$ and $v \in V(G) \setminus V_i$, we have 
	\begin{align}
		\label{equ:pVi*=max}
		\deg_\G(v,V_i^*) 
		\leq \deg_\G(v,V_i')  
		\overset{\text{\ref{itm:b3}}}{\leq} (1+ \eps ) p |V_i'| 
		\le  (1+ \eps ) p |V_i|
		\overset{\text{\eqref{equ:V*-V}}}{\leq} 3 p |V_i^*|/2.
	\end{align}

\begin{claim}
	$(G,\cV^*)$ is an $(\eps_{\ref{lem:blow-up}},1/r,p)$-regular $R$-partition, which is also $(\eps_{\ref{lem:blow-up}},1/r,p)$-super-regular on~$R'$.
\end{claim}
\begin{proofclaim}
	Recall that, for any $ij \in E(R)$, the pair $(V_i,V_j)$ is $(\eps,1/r,p)$-lower-regular in~$G'$ and $\eps \leq \eps_{\ref{lem:blow-up}}/2$.
	So Lemma~\ref{lem:slice} together with~\eqref{equ:V*-V} implies that $(V_i^*,V_j^*)$ is $(\eps_{\ref{lem:blow-up}},1/r,p)$-lower-regular in~$G'$.
	Hence $(G,\cV^*)$ is an $(\eps_{\ref{lem:blow-up}},1/r,p)$-regular $R$-partition.
	
Consider any $ij \in E(R')$ and $v \in V^*_i$. 
Then
\begin{align*}
\deg_{G}(v,V_j^*) & \ge  \deg_{G}(v,V_j') - \deg_{\G}(v,V_j' \cap U) - |U^+|\\
 & \overset{ \mathclap{\text{\ref{itm:b2}, \ref{itm:deg-U}, \eqref{equ:size-W'} }} }{ \ge }
( 1 / r - \eps - 2^{-9} - 2^{-10} ) p |V_j'|
 \ge ( 1/r  -  \eps_{\ref{lem:blow-up}} ) p |V_j^*|\\
& \overset{ \mathclap{ \text{\eqref{equ:pVi*=max} } }}{ \ge } ( 1/r - \eps_{\ref{lem:blow-up}}) \max\{p |V_i^*|, \deg_{\G}(v,V_j^*)/2\}\nc{}{.}
\end{align*}
Therefore, $(G,\cV^*)$ is $(\eps_{\ref{lem:blow-up}},1/r,p)$-super-regular on~$R'$.
\end{proofclaim}

\begin{claim}
	$(G,\cV^*)$ has one-sided $(\eps_{\ref{lem:blow-up}},1/r,p)$-inheritance on~$R'$.
\end{claim}
\begin{proofclaim}
	Fix $ij \in E(R')$ and $z\in  V(\G) \setminus (W \cup V_i')$.
	Note that 
	\begin{align*}
		\deg_\G(z,V_i^*)& \ge  \deg_\G(z,V_i') - \deg_{\G}(z,V_i' \cap U) - |U^+|
	\\ &
	\overset{ \mathclap{ \text{\ref{itm:b3}, \ref{itm:deg-U}, \eqref{equ:size-W'} }} }{ \ge }	
	( 1  - \eps- 2^{-9} - 2^{-10} ) p |V_i'|
	\ge 3 p |V_i^*|/4
	\overset{ \mathclap{ \text{\eqref{equ:pVi*=max} } }}{\ge} \deg_\G(z,\nc{V_i^*}{V_i'})/2.
	\end{align*}
	In particular, $\deg_\G(z,V_i^*) \geq  \deg_\G(z, {V_i'})/2$ for every $z \in V_j^*$.
	By~\eqref{equ:V*-V}, $|V_i^*| \geq |V_i'|/2$.
	Recall {that by~{\ref{itm:b4}}} the pair $\big(N_\G(z,V_i'),V_j' \big)$ is $(\eps_{\ref{lem:blow-up}}/2,1/r,p)$-lower-regular in~$G'$.
	So Lemma~\ref{lem:slice} implies that $\big(N_\G(z,V_i^*),V_j^*\big)$ is $(\eps_{\ref{lem:blow-up}},1/r,p)$-lower-regular in $G'$ as well as in~$G$, which yields the desired inheritance.
\end{proofclaim}
This finishes the proof of Lemma~\ref{lem:approximate-partition}.

\section*{Acknowledgements}
We would like to thank Daniela Kühn, Deryk Osthus and the anonymous referees for their helpful comments and Louis DeBiasio for stimulating discussions.


\providecommand{\bysame}{\leavevmode\hbox to3em{\hrulefill}\thinspace}
\providecommand{\MR}{\relax\ifhmode\unskip\space\fi MR }
\providecommand{\MRhref}[2]{%
	\href{http://www.ams.org/mathscinet-getitem?mr=#1}{#2}
}
\providecommand{\href}[2]{#2}

\end{document}